\documentclass[12pt]{article}
\usepackage{amssymb}
\usepackage{mathrsfs}
\usepackage[hypertex]{hyperref}
\usepackage{amsfonts}
\usepackage{graphicx}
\usepackage{mathptmx}
\usepackage{latexsym,amsmath,amssymb,amsfonts,amsthm}
\newcommand{\R}{\mathbb{R}}

\newcommand{\bcen}{\begin{center}}
\newcommand{\ecen}{\end{center}}
\newtheorem{theorem}{Theorem}[section]
\newtheorem{lemma}[theorem]{Lemma}

\newtheorem{remark}[theorem]{Remark}

\def\det{{\rm det}}
\def\Ric{{\rm Ric}}

\def\p{\partial}
\def\n{\nabla}

\def\R{\mathbb{R}}

\def\si{\sum_{i=1}^{n}}
\def\sj{\sum_{j=1}^{n}}
\def\sij{\sum_{i,j=1}^{n}}

\def\O{\Omega}
\def\div{{\mathrm{div}}}
\def\Ric{{\rm Ric}}

\def\L{\mathbb{L_{\phi}}}

\def\dm{dv_{\phi}}
\def\dn{dv}

\def\ino{\int_\Omega}
\def\inpo{\int_{\partial\Omega}}

\def\D{\Delta}

\def\({\left(}
\def\){\right)}

\begin{document}
\setcounter{page}{1}
\title{Sharp eigenvalue estimates and related rigidity theorems}
\author{Yanlin Deng$^{\dag}$, Feng Du$^{\dag}$, Jing Mao$^{\dag,\ddag,\ast}$, Yan Zhao$^{\ddag}$}

\date{}
\protect\footnotetext{\!\!\!\!\!\!\!\!\!\!\!\!{$^{\ast}$Corresponding author}
 }
\maketitle ~~~\\[-15mm]

\begin{center}
{\footnotesize   $^{\dag}$School of Mathematics and Physics Science,\\
Jingchu University of Technology, Jingmen, 448000, China\\
$^{\ddag}$Faculty of Mathematics and Statistics,\\
 Key Laboratory of Applied
Mathematics of Hubei Province, \\
Hubei University, Wuhan 430062, China\\
Key Laboratory of Intelligent Sensing System and Security (Hubei
University), Ministry of Education\\
 Emails: dyl690926@163.com (Y. L. Deng), \\ defengdu123@163.com (F. Du),
  jiner120@163.com (J. Mao). }
\end{center}


\begin{abstract}
In this paper, sharp bounds for the first nonzero eigenvalues of
different type have been obtained. Moreover, when those bounds are
achieved, related rigidities can be characterized. More precisely,
first, by applying the Bishop-type volume comparison proven in
\cite{fmi,m1} and the Escobar-type eigenvalue comparisons for the
first nonzero Steklov eigenvalue of the Laplacian proven in
\cite{ywmd}, for manifolds with \emph{radial} sectional curvature
upper bound, under suitable preconditions, we can show that the
first nonzero Wentzell eigenvalue of the geodesic ball on these
manifolds can be bounded from above by that of the geodesic ball
with the same radius in the model space (i.e., spherically symmetric
manifolds) determined by the curvature bound. Besides, this upper
bound for the first nonzero Wentzell eigenvalue can be achieved if
and only if these two geodesic balls are isometric with each other.
This conclusion can be seen as an extension of eigenvalue
comparisons in \cite{je1,ywmd}.
 Second, we prove a general Reilly formula for the drifting
 Laplacian, and then use the formula to give a \emph{sharp} lower bound
for the first nonzero Steklov eigenvalue of the drifting
 Laplacian on compact smooth metric measure
spaces with boundary and convex potential function. Besides, this
lower bound can be achieved only for the Euclidean ball of the
prescribed radius.
 \end{abstract}

\markright{\sl\hfill Y. L. Deng, F. Du, J. Mao, Y. Zhao \hfill}

\section{Introduction}
\renewcommand{\thesection}{\arabic{section}}
\renewcommand{\theequation}{\thesection.\arabic{equation}}
\setcounter{equation}{0} \label{intro}

Throughout this paper, assume that $(M,g)$ is an $n$-dimensional ($
n\geq 2 $) complete Riemannian manifold with the metric $g$. Let
$\Omega\subseteq M$ be a compact domain with smooth\footnote{We do
know that maybe \emph{smoothness} assumption is a little bit
stronger for the eigenvalue problem (\ref{eigen-1}). For instance,
when $\beta=0$, (\ref{eigen-1}) degenerates into the classical
Steklov eigenvalue problem (\ref{eigen-2}) below, and in this
situation (\ref{eigen-2}) only has discrete spectrum even if the
boundary $\partial\Omega$ is only Lipschitz continuous (not
necessary to be smooth) -- see \cite[Theorem 6.2]{JN} for details.
Based on this fact, when considering different eigenvalue problems
in this paper, in order to avoid putting two much attention on the
discussion of regularity assumptions for the boundary, we always
assume that the boundary (if exists) is smooth.} boundary
$\partial\Omega$. Denote by $\Delta$ and $\overline{\Delta}$ the
Laplace-Beltrami operators on $\Omega$ and $\partial\Omega$,
respectively. Consider the
 eigenvalue problem with the Wentzell boundary condition as follows
\begin{eqnarray} \label{eigen-1}
\left\{
\begin{array}{ll}
\Delta u=0\qquad & \mathrm{in}~\Omega, \\[0.5mm]
-\beta{\overline{\Delta}}u+\frac{\partial
u}{\partial\vec{\eta}}={\tau u} \qquad &
\mathrm{on}~{\partial\Omega},
\end{array}
\right.
\end{eqnarray}
where $\vec{\eta}$ is the outward unit normal vector of the boundary
$\partial\Omega$, and $\beta$ is a given real number. The boundary
value problem (\ref{eigen-1}) is called \emph{the Wentzell
eigenvalue problem of the Laplacian}. It is known that for
$\beta\geq0$, the eigenvalue problem (\ref{eigen-1}) only has the
discrete spectrum and its elements, called \emph{eigenvalues}, can
be listed non-decreasingly as follows
\begin{eqnarray*}
0=\tau_{0}(\Omega)<\tau_{1}(\Omega)\leq\tau_{2}(\Omega)\leq\tau_{3}(\Omega)\leq\cdots\uparrow\infty.
\end{eqnarray*}
Besides, by the variational principle, it is not hard to know that
the first nonzero eigenvalue $\tau_{1}(\Omega)$ of (\ref{eigen-1})
can be characterized as follows
 \begin{eqnarray} \label{cha-1}
&&\tau_{1}(\Omega)=\min\Bigg{\{}\frac{\int_{\Omega}|\nabla
u|^{2}dv+\beta\int_{\partial\Omega}|\overline{\nabla}u|^{2}dA}{\int_{\partial\Omega}u^{2}dA}\Bigg{|}\nonumber\\
&& \qquad \qquad \qquad u\in
W^{1,2}(\Omega),\mathrm{Tr}_{\partial\Omega}(u)\in
W^{1,2}(\partial\Omega),u\neq0,\int_{\partial\Omega}udA=0\Bigg{\}},
 \end{eqnarray}
where $\nabla$, $\overline{\nabla}$ are the gradient operators on
$\Omega$ and $\partial\Omega$ respectively,
$\mathrm{Tr}_{\partial\Omega}$ is the trace operator,
$W^{1,2}(\Omega)$ is the completion of the set of smooth functions
$C^{\infty}(\Omega)$ under the Sobolev norm
$\|u\|_{1,2}=\left(\int_{\Omega}u^{2}dv+\int_{\Omega}|\nabla
u|^{2}dv\right)^{1/2}$, and $W^{1,2}(\partial\Omega)$ is defined
similarly. Here, $dv$, $dA$ are volume elements of the domain
$\Omega$ and its boundary $\partial\Omega$, respectively.  For the
eigenvalue problem (\ref{eigen-1}), there are some interesting
estimates for eigenvalues $\tau_{i}$ recently -- see, e.g.,
\cite{dkl,dmwx,dwx,xw1}.

Clearly, when $\beta=0$, the eigenvalue problem (\ref{eigen-1})
degenerates into the following classical Steklov eigenvalue problem
\begin{eqnarray} \label{eigen-2}
\left\{
\begin{array}{ll}
\Delta u=0\qquad & \mathrm{in}~\Omega, \\[0.5mm]
\frac{\partial u}{\partial\vec{\eta}}={p u} \qquad &
\mathrm{on}~{\partial\Omega},
\end{array}
\right.
\end{eqnarray}
which only has discrete spectrum and all the eigenvalues can be
listed non-decreasingly as follows
\begin{eqnarray*}
0=p_{0}(\Omega)<p_{1}(\Omega)\leq p_{2}(\Omega)\leq
p_{3}(\Omega)\leq\cdots\uparrow\infty.
\end{eqnarray*}
Besides, the first nonzero Steklov eigenvalue can be characterized
as follows (see, e.g., \cite[p. 144]{je1})
\begin{eqnarray} \label{cha-2}
p_{1}(\Omega)=\min\left\{\frac{\int_{\Omega}|\nabla
u|^{2}dv}{\int_{\partial\Omega}u^{2}dA}\Bigg{|}u\in
W^{1,2}(\Omega),u\neq0,\int_{\partial\Omega}udA=0\right\}.
 \end{eqnarray}
By  (\ref{cha-2}), one can easily get the Sobolev trace inequality,
which makes an important role in the study of existence and
regularity of solutions of some boundary value problems, as follows
\begin{eqnarray*}
\int_{\partial\Omega}|u-u_{0}|^{2}dA\leq\frac{1}{p_{1}(\Omega)}\int_{\Omega}|\nabla
u|^{2}dv,
\end{eqnarray*}
where $u_{0}$ is the mean value of the function $u$ when restricted
to the boundary.

By (\ref{cha-1}) and (\ref{cha-2}), it is not hard to get the fact:
\begin{itemize}

\item \textbf{Fact 1}. \label{fact-1}
 \emph{For $\beta>0$, one has
\begin{eqnarray*}
\tau_{1}(\Omega)\geq\beta\lambda^{c}_{1}(\partial\Omega)+p_{1}(\Omega),
\end{eqnarray*}
where $\lambda^{c}_{1}(\partial\Omega)$ denotes the first nonzero
closed eigenvalue of the Laplacian on the boundary $\partial\Omega$.
Moreover, the equality can be obtained if and only if any
eigenfunction $u$ of $\tau_{1}(\Omega)$ is
 also the eigenfunction corresponding  to $p_{1}(\Omega)$ and $u|_{\partial
 \Omega}$ is the eigenfunction corresponding  to $\lambda_{1}^{c}(\partial\Omega)$ on $\partial
 \Omega$.}

\end{itemize}
Combining the Bishop-type volume comparison (see \cite[Theorem
4.2]{fmi} or \cite[Theorem 2.3.2]{m1}) with the Escobar-type
eigenvalue comparisons for the first nonzero Steklov eigenvalue of
the Laplacian (see \cite[Theorems 1.5 and 1.6]{ywmd}), we can get a
comparison for the first nonzero Wentzell eigenvalue $\tau_{1}$ of
the Laplacian on complete manifolds with radial sectional curvature
bounded from above. More precisely, we have:

\begin{theorem} \label{maintheorem-1}
Assume that $(M,g)$ is an $n$-dimensional complete Riemannian
manifold having a radial sectional curvature upper bound $k(t)$
w.r.t. $p$, where $t:=d(p,\cdot)$ denotes the distance to the point
$p\in M$, and $k(t)$ is a continuous function w.r.t. $t$. Let
$B(p,r)$ be the geodesic ball, with center $p$ and radius $r$, on
$M$. For the Wentzell eigenvalue problem (\ref{eigen-1}) with
$\beta\geq0$, we have
\begin{itemize}

\item If $n=2,3$, then
\begin{eqnarray} \label{egc-1}
\tau_{1}(B(p,r))\leq\tau_{1}(\mathcal{B}(p^{+},r)),
\end{eqnarray}
 where $r<\min\{\mathrm{inj}(p),l\}$ with $\mathrm{inj}(p)$ the injectivity radius at $p$, and $\mathcal{B}(p^{+},r)$ is the geodesic
 ball, with center $p^{+}$ and radius $r$, of the spherically
 symmetric manifold\footnote{~In this paper, as usual, $\mathbb{S}^{n-1}$ stands for the $(n-1)$-dimensional Euclidean unit sphere.} $M^{+}=[0,l)\times_{f}\mathbb{S}^{n-1}$ with the
 base point $p^{+}$ and the warping function $f$ determined by
 \begin{eqnarray}\label{ODE-1}
\left\{\begin{array}{lll}
 f''(t)+k(t)f(t)=0& \qquad \mathrm{on}~(0,l),\\
 f'(0)=1,~ f(0)=0,&\\
 f|_{(0,l)}>0.
\end{array}\right.
\end{eqnarray}
  Equality in (\ref{egc-1}) holds if and only if $B(p,r)$ is
isometric to $\mathcal{B}(p^{+},r)$.

\item If $n\geq4$ and furthermore the first nonzero closed
eigenvalues of the Laplacian on the boundary satisfy
\begin{eqnarray} \label{pre-1}
\lambda_{1}^{c}(\partial
B(p,r))\leq\lambda_{1}^{c}(\partial\mathcal{B}(p^{+},r)),
\end{eqnarray}
then the same conclusion as in the lower dimensional cases $n=2$ and
$n=3$ can also be obtained.
\end{itemize}
\end{theorem}

\begin{remark}
\rm{
 (1) As we know, roughly speaking, the main contribution of the
Escobar-type eigenvalue comparison theorems (\cite[Theorems 1.5 and
1.6]{ywmd}) for the first nonzero Steklov eigenvalue of the
Laplacian is that J. Mao and his collaborators therein have
successfully \emph{weakened} the curvature assumption for the
classical eigenvalue comparison theorem (of the first nonzero
Steklov eigenvalue of the Laplacian) in J. F. Escobar's influential
work \cite{je1} -- ``\emph{the sectional curvature is bounded from
above by some constant}" has been weakened to ``\emph{the
\textbf{radial} sectional curvature has an upper bound (w.r.t. the
chosen point) which is given by a continuous function of the
Riemannian distance parameter}", and correspondingly, this change
leads to the situation that the model spaces in Escobar's setting,
which are space forms, should be replaced by more general model
spaces (i.e., spherically symmetric manifolds) in Mao's setting. The
corresponding author here, Prof. J. Mao, has used spherically
symmetric manifolds as the model spaces to get some interesting
(volume, eigenvalue, heat kernel) comparison conclusions -- see
\cite{fmi,m1,m2,m3,m4,m5,ywmd} for details.

Clearly, if $\beta=0$, then the eigenvalue problem (\ref{eigen-1})
degenerates into the classical Steklov eigenvalue problem
(\ref{eigen-2}) of the Laplacian. Correspondingly, our Theorem
\ref{maintheorem-1} becomes exactly the eigenvalue comparisons
\cite[Theorems 1.5 and 1.6]{ywmd} directly.

\textbf{In sum}, Escobar \cite{je1} has obtained Theorem
\ref{maintheorem-1} for $\beta=0$ under an upper sectional curvature
bound by a constant. In \cite{ywmd}, it has been generalized under
the variable upper sectional curvature bound. We stress that Theorem
\ref{maintheorem-1} for $\beta>0$ is \textbf{new} even in the case
where the sectional curvature is bounded from above by a constant.
\\
(2) People are really caring about the case $\beta>0$. The proof of
our Theorem \ref{maintheorem-1} here deeply depends on some
arguments in the proof of \cite[Theorems 1.5 and 1.6]{ywmd}, but if
one checks the proof of Theorem \ref{maintheorem-1} carefully, then
one would find that the eigenvalue comparison here cannot be
obtained by applying the Escobar-type eigenvalue comparisons (for
the first nonzero Steklov eigenvalue of the Laplacian)
\cite[Theorems 1.5 and 1.6]{ywmd} directly. In fact, for complete
$n$-manifolds whose radial sectional curvature has an upper bound
$k(t)$ w.r.t. $p$, by \cite[Theorems 1.5 and 1.6]{ywmd}, one has
$p_{1}(B(p,r))\leq p_{1}(\mathcal{B}(p^{+},r))$ with
$r<\min\{\mathrm{inj}(p),l\}$, but together with (\ref{2-17}), one
cannot get $\tau_{1}(B(p,r))\leq\tau_{1}(\mathcal{B}(p^{+},r))$
since $p_{1}(B(p,r))+\beta\lambda_{1}^{c}(\partial
B(p,r))\leq\tau_{1}(B(p,r))$. On contrary, for complete
$n$-manifolds whose radial sectional curvature has an upper bound
$k(t)$ w.r.t. $p$,  by Theorem \ref{maintheorem-1}, one has
$\tau_{1}(B(p,r))\leq\tau_{1}(\mathcal{B}(p^{+},r))$ with
$r<\min\{\mathrm{inj}(p),l\}$, which together with (\ref{2-17})
implies
 \begin{eqnarray*}
p_{1}(B(p,r))+\beta\lambda_{1}^{c}(\partial
B(p,r))\leq\tau_{1}(B(p,r))\leq\tau_{1}(\mathcal{B}(p^{+},r))=p_{1}(\mathcal{B}(p^{+},r))+\beta\lambda_{1}^{c}(\partial\mathcal{B}(p^{+},r)).
 \end{eqnarray*}
 Therefore, one finally gets
 \begin{eqnarray*}
p_{1}(B(p,r))\leq p_{1}(\mathcal{B}(p^{+},r))
 \end{eqnarray*}
  since the precondition $\lambda_{1}^{c}(\partial
B(p,r))\leq\lambda_{1}^{c}(\partial\mathcal{B}(p^{+},r))$ is valid
in this setting. So, our Theorem \ref{maintheorem-1} is a very
interesting improvement of \cite[Theorems 1.5 and 1.6]{ywmd}, and of
course covers it as a special case.   }
\end{remark}

For a given complete $n$-dimensional Riemannian manifold $(M,g)$,
the triple $(M,g,e^{-\phi}dv)$ is called a smooth metric measure
space (SMMS for short), where $\phi$ is a \emph{smooth real-valued}
function defined on $M$. We call $\dm:=e^{-\phi}dv$ the weighted
volume density (also called the weighted Riemannian density). On a
SMMS $(M,g,e^{-\phi}dv)$, we can define the so-called \emph{drifting
Laplacian} (also called \emph{weighted Laplacian} or
\emph{Witten-Laplacian}) $\L$ as follows
\begin{eqnarray*}
\L:=\Delta-g(\nabla \phi,\nabla\cdot),
\end{eqnarray*}
where, as before, $\nabla$ and $\Delta$ are the gradient operator
and the Laplace operator on $M$, respectively. Define the shape
operator $S$ of $\p \O$ as $S(X) = \n_X\vec{\eta}$, and then the
second fundamental form of $\p \O$ is defined as $II(X,Y) = \langle
S(X),Y\rangle$, where $X,Y \in T\p \O$ with $T\p \O$ the tangent
bundle of $\O$. The eigenvalues of $S$ are called the principal
curvatures of $\p \O$ and the mean curvature $H$ of $\p \O$ is given
by $H =\frac{1}{n-1} \mathrm{tr} S$, where $\mathrm{tr} S$ denotes
the trace of $S$. In 2010, Ma and Du \cite[Theorem 1]{MD} gave a
Reilly-type formula for the weighted Laplacian. In fact, for $f\in
C^{\infty}(\Omega)$, they have proven the following Reilly-type
formula
\begin{eqnarray}\label{c1-1}
&&\ino \((\L f)^2-|\n^2 f|^2-\mathrm{Ric}^\phi(\n f, \n f)\)\dm
=\nonumber \\
&&\qquad \qquad \qquad \inpo\((n-1)H^\phi
u^2+2u\overline{\L}z+II(\overline{\n}z,\overline{\n}z)\) dA_\phi,
\end{eqnarray}
where $u =\frac{\p f}{\p\vec{\eta}} |_{\p \O}$, $z = f|_{\p \O} $,
$\n^{2}f$ is the Hessian of $f$, $H^\phi=H+\frac{1}{n-1}\frac{\p
\phi}{\p\vec{\eta}}$ denotes the $\phi$-mean curvature\footnote{
Readers might find that the second term $\frac{1}{n-1}\frac{\p
\phi}{\p\vec{\eta}}$ of the $\phi$-mean curvature has different
forms in literatures (for instance, $-\frac{1}{n-1}\frac{\p
\phi}{\p\vec{\eta}}$, $\frac{\p \phi}{\p\vec{\eta}}$, etc), but
actually they have no essential difference. } (also called the
weighted mean curvature, see, e.g., \cite{WW} for this notion),
$\mathrm{Ric}^{\phi}:=\mathrm{Ric}+\n^{2} \phi$ denotes the
Bakry-\'Emery Ricci curvature tensor of $M$ with
$\mathrm{Ric}(\cdot,\cdot)$ the Ricci curvature,
$dA_\phi:=e^{-\phi}dA$ is the induced Riemannian density of the
boundary, $\overline{\L}$ is the drifting Laplacian on
$\partial\Omega$ related to $\overline{\Delta}$. Clearly, if
$\phi=const.$ is a constant function, then (\ref{c1-1}) becomes the
following celebrated  Reilly formula
\begin{eqnarray} \label{rc}
\ino \((\Delta f)^2-|\n^2 f|^2-\mathrm{Ric}(\n f, \n f)\)dv
=\inpo\((n-1)H u^2+2u\overline{\Delta}z
+II(\overline{\n}z,\overline{\n}z)\)dA,
\end{eqnarray}
which was firstly proven by R. Reilly \cite{R} in 1970s. Reilly's
formula is a useful tool for eigenvalue estimates. For instance,
Reilly \cite{R} used the formula to prove a Lichnerowicz type sharp
lower bound for the first eigenvalue of the Laplacian on manifolds
with boundary. By applying (\ref{rc}), Escobar \cite{je3}, Wang-Xia
\cite{WX} successfully gave some estimates for the first nonzero
Steklov eigenvalue of the Laplacian, respectively. Qiu and Xia
\cite{QX} extended Reilly's formula to the following version:
\begin{eqnarray} \label{qx-1}
\nonumber&&\ino V\(\(\Delta f+Kn f\)^2-|\n^2 f+K f g|^2\)dv
\\\nonumber&=&\inpo V\(2u \overline{\Delta} z+(n-1)H u^2+II(\overline{\n}z,\overline{\n}z)+(2n-2)Kuz\)dA\\\nonumber&&+\inpo \frac{\p V}{\p\vec{\eta} }\(|\overline{\n} z|^2-(n-1)Kz^2\)dA+\ino (n-1)\(K\D V +nK^2 V\)f^2dv\\&&+\ino \(\n^2 V-\D V g-(2n-2)KV g+V\Ric\)\(\n f,\n
f\)dv,
\end{eqnarray}
where $K\in\mathbb{R}$, $V:\overline{\Omega}\rightarrow\mathbb{R}$
is a given a.e. twice differentiable function, and other notations
have the same meanings as before. Recently, by applying this
generalized Reilly's formula (\ref{qx-1}), under the nonnegative
sectional curvature assumption, Xia and Xiong \cite{XX} can obtain a
sharp lower bound estimate for the first nonzero Steklov eigenvalue
of the Laplacian, which gives a partial answer to the following
\textbf{Escobar's conjecture}:
\begin{itemize}

\item (\cite{je2}) Let $(N^{n},\widetilde{g})$ be a compact Riemannian manifold with
boundary and dimension $n\geq3$. Assume that
$\mathrm{Ric}(\widetilde{g})\geq0$ and that the second fundamental
form $II$ satisfies $II\geq c\mathrm{I}$ on $\partial N^{n}$, $c>0$.
Then
\begin{eqnarray*}
p_{1}(N^n)\geq c,
\end{eqnarray*}
and the equality holds only for the Euclidean ball of radius
$\frac{1}{c}$.
\end{itemize}

We can prove the following more general Reilly-type formula for the
drifting Laplacian.
\begin{theorem}  \label{theorem1}
Let $V:\overline{\Omega}\rightarrow \R$ be a a.e. twice differential
function, where $\O$ is a bounded domain with smooth boundary
$\p\Omega$ on an $n$-dimensional SMMS $(M,g,e^{-\phi}dv)$, $n\geq2$.
Given a smooth function $f$ on $\Omega$, we have
\begin{eqnarray} \label{rf-11}
\nonumber&&\ino V\(\(\L f+Kn f\)^2-|\n^2 f+K f g|^2+2K f\langle\n f,\n \phi\rangle\)\dm
\\\nonumber&=&\inpo V\(2u \overline{\L} z+(n-1)H^\phi u^2+II(\overline{\n}z,\overline{\n}z)+(2n-2)Kuz\)dA_\phi\\\nonumber&&+\inpo \frac{\p V}{\p\vec{\eta} }\(|\overline{\n} z|^2-(n-1)Kz^2\)dA_\phi+\ino (n-1)\(K\L V +nK^2 V\)f^2\dm\\&&+\ino \(\n^2 V-\L V g-(2n-2)KV g+V\Ric^\phi\)\(\n f,\n
f\)\dm,
\end{eqnarray}
where same notations have the same meanings as those in
(\ref{c1-1}), (\ref{rc}) and (\ref{qx-1}).
\end{theorem}

\begin{remark}
\rm{(1) Clearly, if $\phi=const.$, our Reilly-type formula
(\ref{rf-11}) degenerates into (\ref{qx-1}); if $V\equiv1$ and
$K=0$, (\ref{rf-11}) becomes (\ref{c1-1}); if $V\equiv1$, $K=0$ and
$\phi=const.$, our formula (\ref{rf-11}) degenerates into the
classical Reilly's formula (\ref{rc}). \\
 (2) In the sequel, one might see that two applications of (\ref{rf-11}) on
 eigenvalue estimation have been shown. In fact, they are:
\begin{itemize}
\item By applying (\ref{rf-11})  and choosing $V\equiv1$ and
$K=0$, we can get a lower bound estimate
$\sigma_{1,\phi}(\Omega)>c/2$ for the first nonzero eigenvalue
$\sigma_{1,\phi}(\Omega)$ of the Steklov-type eigenvalue problem
(\ref{a10}) below if for the compact Riemannian manifold $\Omega$,
$\mathrm{Ric}^{\phi}\geq0$, the principle curvatures of its boundary
$\partial\Omega$ have a positive lower bound $c>0$, and
$H^{\phi}>c$. Especially, if $\phi=const.$, this lower bound
estimate conclusion gives a partial answer to the above
\textbf{Escobar's conjecture}. For details, see (2) of Remark
\ref{remark 1-6}.

\item For a given compact Riemannian manifold $\Omega$ whose smooth boundary
$\partial\Omega$ has a positive lower bound $c>0$ for its principal
curvatures, if the curvature pinching assumptions in Theorem
\ref{theorem2} were satisfied, by applying (\ref{rf-11})  and
choosing $K=0$ and $V=\rho-\frac{c}{2}\rho^2$ with
$\rho=\mathrm{dist}(\cdot,\p \Omega)$ the distance function to the
boundary $\p\Omega$, we can get an eigenvalue estimate
$\sigma_{1,\phi}(\Omega)\geq c$ and a rigidity characterization for
the equality case in this estimate.

\end{itemize}
Is it possible to find applications for the Reilly-type formula
(\ref{rf-11}) in the case $K\neq0$ and nonconstant $V$? The work
\cite{QX} has given an affirmative answer to this question. In fact,
as we have pointed out, if $\phi=const.$, our Reilly-type formula
(\ref{rf-11}) degenerates into (\ref{qx-1}), and in \cite{QX}, by
choosing further $K=1$ and $V(x)=\cos r(x)$ (or $K=-1$ and
$V(x)=\cosh r(x)$) in (\ref{qx-1}), Qiu and Xia can give an
alternative proof of the celebrated Alexandrov's theorem in the unit
hemisphere $\mathbb{S}^{n}_{+}$ and the hyperbolic space
$\mathbb{H}^{n}$ with constant sectional curvature $-1$, where
$r(x)=\mathrm{dist}(o,x)$ denotes the radial distance (in
$\mathbb{S}^{n}_{+}$ or $\mathbb{H}^{n}$) from the origin $o$ to the
point $x$. Moreover, by choosing $K=-1$ and $V(x)=\cosh r(x)$ in
(\ref{qx-1}) and using the classical Hessian comparison theorem,
they \cite[Theorem 1.3]{QX} can give a Heintze-Karcher type
inequality for compact Riemannian manifolds (with smooth boundary)
whose sectional curvature is bounded from below by $-1$. It is
feasible to improve Qiu-Xia's main conclusions in \cite{QX} to the
nonconstant weighted case (i.e. $\phi\neq const.$), and we wish to
leave this attempt to readers who are interested in this topic.
 }
\end{remark}

Consider the Steklov-type eigenvalue problem of the drifting
Laplacian on a compact domain $\Omega$ of an $n$-dimensional
($n\geq2$) SMMS $(M,g,e^{-\phi}dv)$ as follows
\begin{eqnarray}\label{a10}
\left\{\begin{array}{ccc} \L u
=0&&~~\mbox{in} ~~\Omega, \\[2mm]
\frac{\p u}{\p \vec{\eta}}=\sigma u&&~~~~~\mbox{on}~~\partial
\Omega,
\end{array}\right.
\end{eqnarray}
where notations have the same meanings as before. One can get that
(\ref{a10}) has a discrete spectrum and all the eigenvalues can be
listed non-decreasingly as follows
\begin{eqnarray*}
0=\sigma_{0,\phi}(\Omega)<\sigma_{1,\phi}(\Omega)\leq
\sigma_{2,\phi}(\Omega)\leq
\sigma_{3,\phi}(\Omega)\leq\cdots\uparrow\infty.
\end{eqnarray*}
Besides, the first nonzero Steklov eigenvalue
$\sigma_{1,\phi}(\Omega)$ of $\L$ can be characterized as follows
\begin{eqnarray} \label{cha-3}
\sigma_{1,\phi}(\Omega)=\min\left\{\frac{\int_{\Omega}|\nabla
u|^{2}dv_{\phi}}{\int_{\partial\Omega}u^{2}dA_{\phi}}\Bigg{|}u\in
\widetilde{W}^{1,2}(\Omega),u\neq0,\int_{\partial\Omega}udA_{\phi}=0\right\},
 \end{eqnarray}
where $\widetilde{W}^{1,2}(\Omega)$ is the completion of the set of
smooth functions $C^{\infty}(\Omega)$ under the weighted Sobolev
norm
${\widetilde{\|u\|}}_{1,2}=\left(\int_{\Omega}u^{2}dv_{\phi}+\int_{\partial\Omega}|\nabla
u|^{2}dA_{\phi}\right)^{1/2}$. Please see \textbf{Appendix A}  for
the explanation of the spectral structure of the eigenvalue problem
(\ref{a10}). Recently, the eigenvalue problem (\ref{a10}) (with $M$
chosen to be the $n$-dimensional Euclidean space or hyperbolic
space) has also been considered in our another work \cite{MZ}, and
interesting Brock-type spectral isoperimetric inequalities can be
obtained -- see \cite[Theorems 1.1 and 1.3]{MZ} for details.

  Applying the Reilly-type formula
(\ref{rf-11}), we can give a sharp lower bound for the first nonzero
Steklov eigenvalue $\sigma_{1,\phi}(\cdot)$ of the drifting
Laplacian. More precisely, we have:

\begin{theorem}  \label{theorem2}
For a given compact domain $\Omega$ of the $n$-dimensional
($n\geq2$) SMMS $(M,g,e^{-\phi}dv)$, assume that the following
curvature pinching conditions
$$\mathrm{Sec}(\cdot,\cdot)\geq0, ~~~~\mathrm{Ric}(\cdot,\cdot)\leq\kappa,~~~~\mathrm{Ric}^{\phi}(\cdot,\cdot)\geq\kappa$$
hold on $\Omega$ for some nonnegative constant $\kappa\geq0$, where
$\mathrm{Sec}$ denotes the sectional curvature tensor of $M$, and
other notations have the same meanings as before. For the eigenvalue
problem (\ref{a10}), if furthermore the principal curvatures of $\p
\Omega$ are bounded below by a constant $c>0$, then we have
\begin{eqnarray} \label{shal}
\sigma_{1,\phi}(\Omega)\geq c,
\end{eqnarray}
with equality holding if and only if $\Omega$ is isometric to an
$n$-dimensional Euclidean ball of radius $\frac{1}{c}$ and
$\n^2\phi=0$.
\end{theorem}

\begin{remark} \label{remark 1-6}
\rm{ (1) Since $\mathrm{Sec}(\cdot,\cdot)\geq0$, one has $\mathrm{Ric}(\cdot,\cdot)\geq0$. Then, together with the assumptions
 $\mathrm{Ric}(\cdot,\cdot)\leq\kappa$, $\mathrm{Ric}^{\phi}(\cdot,\cdot)=\mathrm{Ric}(\cdot,\cdot)+\nabla^{2}\phi\geq\kappa$, it follows that
  $$\nabla^{2}\phi\geq0.$$
  Especially, if $\kappa\equiv0$, then one has $\mathrm{Ric}(\cdot,\cdot)\equiv0$, which implies that
  in this situation $M$ would be a Ricci flat manifold.\\
   (2) We should say that the curvature assumptions
   $\mathrm{Ric}(\cdot,\cdot)\leq\kappa$,
   $\mathrm{Ric}^{\phi}(\cdot,\cdot)\geq\kappa$ on $\Omega$ are
   acceptable. This is because $\Omega$ is compact, and then
   $\mathrm{Ric}(\cdot,\cdot)$ is at least a continuous function
   defined on $\Omega$, which implies that the existence of nonnegative constant $\kappa$ can be assured. In fact, the
   optimal upper bound should be
   $\kappa:=\min_{x\in\Omega}\mathrm{Ric}(v_{x},v_{x})$, with $v_{x}$ an arbitrary unit vector in the tangent space at $x\in\Omega$.
   \\
(3) Clearly, if $\phi=const.$, then the estimate (\ref{shal})
degenerates into $p_{1}(\Omega)\geq c$ and the rigidity also holds.
This assertion is exactly the statement of \cite[Theorem 1]{XX} in
the case $\mathrm{Ric}=\kappa g$ (i.e. $\Omega$ is furthermore an
Einstein manifold). That is to say, if furthermore $\phi=const.$ and
$\Omega$ is an Einstein manifold, our Theorem \ref{theorem2} covers
\cite[Theorem 1]{XX} as a special case. \\
(4) In fact, if $\mathrm{Ric}^{\phi}\geq0$, the principal curvatures
of $\p \Omega$ are bounded below by a constant $c>0$ and
$H^{\phi}>c$, for the eigenvalue problem (\ref{a10}), one can obtain
by applying (\ref{c1-1})
\begin{eqnarray*}
&&0\geq\ino \((\L f)^2-|\n^2 f|^2-\mathrm{Ric}^\phi(\n f, \n f)\)\dm
\\
&&\qquad \qquad \qquad =\inpo\((n-1)H^\phi
u^2+2u\overline{\L}z+II(\overline{\n}z,\overline{\n}z)\) dA_\phi\\
&&\qquad \qquad \qquad > \inpo\left[
-2g(\overline{\n}u,\overline{\n}z)+c\cdot
g(\overline{\n}z,\overline{\n}z)\right] dA_\phi \\
&&\qquad \qquad \qquad = \inpo\left[ -2\sigma_{1,\phi}(\Omega)\cdot
g(\overline{\n}z,\overline{\n}z)+c\cdot
g(\overline{\n}z,\overline{\n}z)\right] dA_\phi,
\end{eqnarray*}
which implies
 \begin{eqnarray} \label{estima-2}
\sigma_{1,\phi}(\Omega)>\frac{c}{2}.
\end{eqnarray}
So, clearly, if $\phi=const.$, the estimate (\ref{estima-2})
degenerates into $p_{1}(\Omega)>c/2$, which is exactly Escobar's
estimate given in \cite[Theorem 8]{je3} for the case $n\geq3$.
Besides, if $\phi=const.$, then $\mathrm{Ric}^{\phi}=\mathrm{Ric}$
and the assumption $II>c\mathrm{I}$ implies $H^{\phi}>c$ directly.
This is because, in this situation,
$H^{\phi}=H=\frac{\mathrm{tr}II}{n-1}>c$. In this sense, our
estimate (\ref{estima-2}) here covers Escobar's conclusion
\cite[Theorem 8]{je3} as a special case and of course gives a
partial answer
to \textbf{Escobar's conjecture}. \\
(5) It is easy to find that our estimate (\ref{estima-2}) here is
covered by the lower bound estimate for the first nonzero eigenvalue
of the Wentzell eigenvalue problem of the drifting Laplacian given
in \cite[Theorem 4.2]{ywmd}. In fact, one only needs to choose
$\beta=0$ in the estimate (4.4) of \cite[Theorem 4.2]{ywmd}, and
then our estimate (\ref{estima-2}) follows directly. The reason why
we do not list our previous result  \cite[Theorem 4.2]{ywmd}
directly is that we would like to show the application of the
Reilly-type formula of the drifting Laplacian (\ref{c1-1})
intuitively.
 }
\end{remark}

The paper is organized as follows. In Section \ref{section2}, for
complete manifolds with radial sectional curvature bounded from
above, the Escobar-type eigenvalue comparison theorem for the first
nonzero Wentzell eigenvalue of the Laplacian on geodesic balls (of
these manifolds) can be set up. Of course, the equality case in this
eigenvalue comparison has been characterized. Besides, a related
\textbf{open problem} will be proposed at the end of this section.
In Section \ref{section3}, we firstly give the proof to the
Reilly-type formula given in Theorem \ref{theorem1}, and then by
applying this Reilly-type formula, under suitable constraints, a
sharp lower bound for the first nonzero Steklov eigenvalue of the
drifting Laplacian can be obtained. Moreover, when this sharp bound
is achieved, a rigidity result can be obtained. That is to say, by
the usage of Theorem \ref{theorem1}, we devote to give the proof of
Theorem \ref{theorem2} in the second part of Section \ref{section3}.

\section{The Escobar-type eigenvalue comparison for the first nonzero Wentzell eigenvalue of the
Laplacian} \label{section2}

We give the proof of Theorem \ref{maintheorem-1} as follows:

\begin{proof} [Proof of Theorem \ref{maintheorem-1}]
Let $\psi(t)$ be the function  satisfying the differential equation
 \begin{eqnarray*}
\left\{\begin{array}{lll}
 \frac{1}{f^{n-1}(t)}\frac{d}{dt}\left(f^{n-1}(t)\frac{d}{dt}\psi(t)\right)-\frac{(n-1)\psi(t)}{f^{2}(t)}=0& \quad \mathrm{in}~(0,l),\\
\psi'(r)=p_{1}(\mathcal{B}(p^{+},r))\psi(r),~ \psi(0)=0,&\\
\end{array}\right.
\end{eqnarray*}
where naturally $p_{1}(\mathcal{B}(p^{+},r))$ is the first non-zero
Steklov eigenvalue of the Laplacian on $\mathcal{B}(p^{+},r)$. As
explained in the proof of \cite[Theorem 1.6]{ywmd} (see page 403 of
\cite{ywmd}), we know that $\psi(t)$ does not change sign on
$(0,r)$. Without loss of generality, one can assume $\psi(t)>0$  on
$(0,r)$, and then $\psi'(t)>0$  on $(0,r)$ since
\begin{eqnarray*}
\psi'(t)=\frac{n-1}{f^{n-1}(t)}\int^{t}_{0}\psi(s)f^{n-3}(s)ds,
\end{eqnarray*}
where $f$ is the solution to (\ref{ODE-1}). Construct the test
function $\varphi(t,\xi)=a_{+}(t)e_{1}(\xi)$, where $e_{1}(\xi)$ is
the eigenfunction of the first nonzero closed eigenvalue
$\lambda_{1}^{c}(\partial B(p,r))$ of the Laplacian on the boundary
$\partial B(p,r)$, and
\begin{eqnarray*}
&& a_{+}(t):=\min\{a(t),0\},\\
 &&
 a(t):=\psi(t)\left[\frac{f^{n-1}(t)}{h(t)}\right]^{1/2}+\int_{t}^{r}\psi(s)\left(\left[\frac{f^{n-1}(s)}{h(s)}\right]^{1/2}\right)'ds
\end{eqnarray*}
with
$h(t):=\max\left\{d^{\ast}(t),\frac{f^{2}(t)}{n-1}d^{\sharp}(t)\right\}$
and
 \begin{eqnarray*}
&& d^{\ast}(t)=\int_{\mathbb{S}^{n-1}}|\nabla
e_{1}|_{\mathbb{S}^{n-1}}^{2}(\xi)J^{n-3}(t,\xi)d\sigma, \\
 && d^{\sharp}(t)=\int_{\mathbb{S}^{n-1}}e_{1}^{2}(\xi)\cdot
 \det\mathbb{A}(t,\xi)d\sigma=\int_{\mathbb{S}^{n-1}}e_{1}^{2}(\xi)\cdot
 \sqrt{|g|}(t,\xi)d\sigma=\int_{\mathbb{S}^{n-1}}e_{1}^{2}(\xi)J^{n-1}(t,\xi)d\sigma.
 \end{eqnarray*}
Here $d\sigma$ denotes the $(n-1)$-dimensional volume element on
$\mathbb{S}^{n-1}$, $\mathbb{A}(t,\xi)$ is the path of linear
transformations (see \cite[Subsection 1.1]{ywmd} for the
definition), and $J^{n-1}=\sqrt{|g|}=\det\mathbb{A}(t,\xi)$
represents the square root of the determinant of the metric matrix.
It is easy to check that $h(t)$ is Lipschitz continuous and hence
differentiable almost everywhere, and moreover $\varphi(t,\xi)\in
W^{1,2}(B(p,r))$. By the characterization (\ref{cha-1}), together
with (3.6)-(3.8) in \cite{ywmd} (see pp. 405-406 of \cite{ywmd}), we
can obtain
 \begin{eqnarray} \label{2-16}
\tau_{1}(B(p,r))&\leq&\frac{\int_{B(p,r)}|\nabla
\varphi|^{2}dv+\beta\int_{\partial
B(p,r)}|\overline{\nabla}\varphi|^{2}dA}{\int_{\partial
B(p,r)}\varphi^{2}dA}\nonumber\\
&\leq&
\frac{\int_{0}^{r}\left(\psi'(t)\right)^{2}f^{n-1}(t)dt+(n-1)\int_{0}^{r}\psi^{2}f^{n-3}(t)dt}{\psi^{2}(r)f^{n-1}(r)}+\beta\cdot\frac{\int_{\partial
B(p,r)}|a_{+}(r)\overline{\nabla}e_{1}(\xi)|^{2}dA}{\int_{\partial
B(p,r)}(a_{+}(r)e_{1}(\xi))^{2}dA}
\nonumber\\
 &=&
p_{1}(\mathcal{B}(p^{+},r))+\beta\cdot\frac{\int_{\partial
B(p,r)}|a_{+}(r)\overline{\nabla}e_{1}(\xi)|^{2}dA}{\int_{\partial
B(p,r)}(a_{+}(r)e_{1}(\xi))^{2}dA}\nonumber\\
&=& p_{1}(\mathcal{B}(p^{+},r))+\beta\lambda_{1}^{c}(\partial
B(p,r)).
 \end{eqnarray}
 Moreover, one can see that the equality in (\ref{2-16}) can
 be attained if $B(p,r)$ is isometric to
 $\mathcal{B}(p^{+},r)$.
As shown in the proof of \cite[Theorem 1.5]{ywmd} (see pp. 406-407
of \cite{ywmd} ), for $n=2,3$, one has $\lambda_{1}^{c}(\partial
B(p,r))\leq\lambda_{1}^{c}(\partial\mathcal{B}(p^{+},r))$ directly.
Putting this fact and the assumption (\ref{pre-1}) (for $n\geq4$)
into (\ref{2-16}), together with \textbf{Fact \ref{fact-1}}, yields
\begin{eqnarray} \label{2-17}
p_{1}(B(p,r))+\beta\lambda_{1}^{c}(\partial
B(p,r))&\leq&\tau_{1}(B(p,r))\nonumber\\
&\leq&
p_{1}(\mathcal{B}(p^{+},r))+\beta\lambda_{1}^{c}(\partial B(p,r))\nonumber\\
&\leq&
p_{1}(\mathcal{B}(p^{+},r))+\beta\lambda_{1}^{c}(\partial\mathcal{B}(p^{+},r))\nonumber\\
&=&\tau_{1}(\mathcal{B}(p^{+},r)).
\end{eqnarray}
When $\tau_{1}(B(p,r))=\tau_{1}(\mathcal{B}(p^{+},r))$, then from
(\ref{2-16}) and (\ref{2-17}) we infer that
 \begin{eqnarray*}
p_{1}(B(p,r))=p_{1}(\mathcal{B}(p^{+},r))  \quad \mathrm{and} \quad
\lambda_{1}^{c}(\partial
B(p,r))=\lambda_{1}^{c}(\partial\mathcal{B}( p^{+},r))
 \end{eqnarray*}
holds, which, by \cite[Theorems 1.5 and 1.6]{ywmd}, implies that
$B(p,r)$ is isometric to $\mathcal{B}(p^{+},r)$. In fact, from the
proof of \cite[Theorem 1.6]{ywmd} (see page 406 of \cite{ywmd}), we
know that $p_{1}(B(p,r))=p_{1}(\mathcal{B}(p^{+},r))$ implies
$J(t,\xi)=f(t)$ on $(0,r)$. Then the rigidity follows by  directly
applying the Bishop-type volume comparison theorem for manifolds
having a radial sectional curvature upper bound (see \cite[Theorem
4.2]{fmi} or \cite[Theorem 2.3.2]{m1}).
\end{proof}

\begin{remark}
\rm{
 (1) The curvature assumption in our Theorem \ref{maintheorem-1} here is reasonable, since for a given complete Riemannian manifold and a chosen
 point onside,
 one can always find a \emph{sharp} upper bound (which is given by a
 continuous function of the distance parameter) for the radial sectional curvature  -- see (2.10) in \cite{fmi} for the accurate expression. \\
 (2) Since some arguments in the proof of \cite[Theorems 1.5 and 1.6]{ywmd} make an important role in the proof of Theorem \ref{maintheorem-1},
 similar to the statement in \cite[Theorems 1.5 and 1.6]{ywmd} and as explained in
\cite[(3) of Remark 1.7]{ywmd}, the restraint on the injectivity
radius is necessary. \\
 (3) Clearly, by the Sturm-Picone separation theorem, if
 $k(t)\leq0$, then the initial value problem (\ref{ODE-1}) has
 positive solution on $(0,\infty)$. More precisely, in this
 situation, $l=\infty$ and $f(t)\geq t$ on $(0,\infty)$. Except the
 non-positivity assumption of  $k(t)$, it is interesting to find
 other assumptions such that (\ref{ODE-1}) has a positive solution
 on $(0,\infty)$. This problem has close relation with the
 oscillation of solutions to the ODE $f''(t)+k(t)f(t)=0$. There
 exist some nice results working on this problem -- see, e.g., Bianchini-Luciano-Marco \cite{blm}, Hille \cite{eh} and Mao \cite[Subsection 2.6]{m1} for
 nice sufficient conditions on $k(t)$ such that (\ref{ODE-1}) has a positive solution
 on $(0,\infty)$.
 }
\end{remark}

Naturally, we can propose:

\vspace{2mm}

\textbf{Open problem}. For $n\geq4$, is the Escobar-type Wentzell
eigenvalue inequality (\ref{egc-1}) also true without the
precondition (\ref{pre-1})?

\section{The Reilly-type formula and its application}
\renewcommand{\thesection}{\arabic{section}}
\renewcommand{\theequation}{\thesection.\arabic{equation}}
\setcounter{equation}{0}  \label{section3}

In this section, we first give the proof of the Reilly-type formula
(\ref{rf-11}), and then show an application of this formula -- the
sharp lower bound estimate (\ref{shal}) with the related rigidity.

Some ideas of the following proof come from  \cite{MD,QX}.

\begin{proof} [Proof of the Theorem \ref{theorem1}] Let $f_i,f_{ij},\cdots$ and $f_\eta$ be covariant derivatives
and the normal derivative of a function $f$ w.r.t. the metric $g$.
Here, an orthonormal tangent frame field
$\{e_{1},e_{2},\cdots,e_{n}\}$ has been used in all the calculations
carried out at some point in $\Omega$. Then we have $\n^2
f=\sum_{i,j=1}^n f_{ij}f_{ij}$. Noticing $\Ric^\phi=\Ric+\n^2\phi$,
we infer from the integration by parts and the Ricci identity that
\begin{eqnarray}\label{b1}
\nonumber&&\ino V|\n^2 f|^2 \dm=\ino V\sij f_{ij}f_{ij}e^{-\phi}\dn\\
\nonumber&=&\inpo V\si f_{i\eta}f_ie^{-\phi}dA-\ino \sij V_j
f_{ij}f_ie^{-\phi}\dn-\ino V\sij
f_{ijj}f_ie^{-\phi}\dn\\\nonumber&&+\ino V\sij
f_{ij}f_i\phi_je^{-\phi}\dn\\\nonumber&=&\inpo V\si
f_{i\eta}f_ie^{-\phi}dA-\ino \sj V_j\(\frac{1}{2}|\n f|^2\)_{j}
e^{-\phi}\dn\
\\\nonumber&&-\ino V\si \((\D f)_i+\sj R_{ij} f_j\)f_ie^{-\phi}\dn+\ino V\sij f_{ij}f_i\phi_je^{-\phi}\dn\\\nonumber&=&\inpo V\si f_{i\eta}f_ie^{-\phi}dA-
\inpo \frac{1}{2} V_\eta |\n f|^2 e^{-\phi}dA+\ino \sj
V_{jj}\frac{1}{2}|\n f|^2 e^{-\phi}\dn\\\nonumber&&-\ino \sj
V_j\phi_j\frac{1}{2}|\n f|^2 e^{-\phi}\dn -\ino V\si \(\(\L f+\sj
f_j \phi_j\)_i+\sj R_{ij} f_j\)f_ie^{-\phi}\dn\\\nonumber&&+\ino
V\sij f_{ij}f_i\phi_je^{-\phi}\dn\\\nonumber&=&\inpo V\si
f_{i\eta}f_ie^{-\phi}dA- \inpo \frac{1}{2} V_\eta |\n f|^2
e^{-\phi}dA+\ino \L V\frac{1}{2}|\n f|^2
e^{-\phi}\dn\\\nonumber&&-\ino V\si \(\(\L f\)_i+\sj \(f_{ji}
\phi_j+f_j\phi_{ji}\)+\sj R_{ij} f_j\)f_ie^{-\phi}\dn+\ino V\sij
f_{ij}f_i\phi_je^{-\phi}\dn\\\nonumber&=&\inpo V\si
f_{i\eta}f_ie^{-\phi}dA- \inpo \frac{1}{2} V_\eta |\n f|^2
e^{-\phi}dA+\ino \L V\frac{1}{2}|\n f|^2
e^{-\phi}\dn\\\nonumber&&\inpo V\L f f_\eta e^{-\phi}dA+\ino \L f\si
V_if_ie^{-\phi}\dn+\ino V\L f\si f_{ii} e^{-\phi}\dn\\\nonumber
&&-\ino V\L f\si f_{i}\phi_i e^{-\phi}\dn-\ino V\sij
\(R_{ij}+\phi_{ij}\)f_if_j e^{-\phi}\dn
\\\nonumber&=&\inpo \(V\si f_{i\eta}f_i-\frac{1}{2} V_\eta |\n f|^2-V\L f f_\eta\) e^{-\phi}dA+\ino \L V\frac{1}{2}|\n f|^2 e^{-\phi}\dn\\&&+\ino \L f\si V_if_ie^{-\phi}\dn+\ino V\(\L f\)^2e^{-\phi}\dn-\ino \Ric^{\phi}\(\n f, \n
f\)e^{-\phi}\dn,
\end{eqnarray}
 where the first usage of the integration by parts was in the second equality of (\ref{b1}), and the Ricci identity
 \begin{eqnarray*}
 f_{jij}-f_{jji}=\sum\limits_{k=1}^{n}f_{k}R_{jkij},
 \end{eqnarray*}
with $R_{jkij}$ the components of the curvature tensor on
$\Omega\subset M$, has been used firstly in the third equality of
(\ref{b1}). Summing the above Ricci identity w.r.t. the index $i$
from $1$ to $n$ yields
 \begin{eqnarray*}
\sum\limits_{j=1}^{n}\left(f_{jij}-f_{jji}\right)=\sum\limits_{k,j=1}^{n}f_{k}R_{jkij},
 \end{eqnarray*}
which directly implies
\begin{eqnarray*}
\sum\limits_{j=1}^{n}f_{ijj}-(\Delta
f)_{i}=\sum\limits_{j=1}f_{j}R_{ji},
\end{eqnarray*}
 with $R_{ji}$ the components of the Ricci tensor $\mathrm{Ric}$ on $\Omega\subset
 M$. The relation (i.e., $\L=\Delta-\langle\nabla\phi,\cdot\rangle$) between the operators $\Delta$ and $\L$ has been
 used several times in the rest part of (\ref{b1}).
 We also infer from the integration by parts that
\begin{eqnarray}\label{b2}
\ino V f \L f e^{-\phi}\dn=\inpo V f f_\eta e^{-\phi}dA-\ino\(V|\n
f|^2+\si f V_i f_i\)e^{-\phi}\dn.
\end{eqnarray}
Combining (\ref{b1}) with (\ref{b2}) yields
\begin{eqnarray}\label{b3}
&&\ino V\(\(\L f+Kn f\)^2-|\n^2 f+K f g|^2\)\dm\nonumber\\&=&\ino
V\((\L f)^2-|\n^2 f|^2\)e^{-\phi}\dn+(2n-2)K\ino V f\L f
e^{-\phi}\dn\nonumber\\&&+n(n-1)K^2\ino V f^2e^{-\phi}\dn-2K\ino V
f\si f_i \phi_i e^{-\phi}\dn\nonumber\\&=&\inpo \(-V\si
f_{i\eta}f_i+\frac{1}{2} V_\eta |\n f|^2+V\L f
f_\eta+(2n-2)KVff_\eta\) e^{-\phi}dA\nonumber\\&&+\ino
\(-\frac{1}{2}\L V|\n f|^2- \L f\si V_if_i+V\Ric^{\phi}\(\n f, \n
f\)\)e^{-\phi}\dn\nonumber\\&& -(2n-2)K\ino\(V|\n f|^2+\si f V_i
f_i\)e^{-\phi}\dn+n(n-1)K^2\ino V
f^2e^{-\phi}\dn\nonumber\\&&-2K\ino V f\si f_i \phi_i e^{-\phi}\dn.
\end{eqnarray}
Using the integration by parts again, we have
\begin{eqnarray}\label{b4}
\nonumber&&\ino-\L f\si V_i f_ie^{-\phi}\dn\nonumber\\
&=&\inpo -f_\eta\si V_i f_i e^{-\phi}dA+\ino \(\sij V_{ij}f_if_j+\si V_i\(\frac{1}{2}|\n f|^2\)_i\)e^{-\phi}\dn\nonumber\\
&=&\inpo \(-f_\eta\si V_i f_i+\frac{1}{2}|\n f|^2V_\eta
\)e^{-\phi}dA+ \nonumber\\
&& \qquad \qquad \ino \(\sij V_{ij}f_if_j-\frac{1}{2}\L V|\n
f|^2\)e^{-\phi}\dn,
\end{eqnarray}
and
\begin{eqnarray}\label{b5}
\nonumber&&\ino \si V_i f_i fe^{-\phi}\dn\\
&=&\ino \si V_i\(\frac{1}{2} f^2\)_ie^{-\phi}\dn \nonumber\\&=&\inpo
\frac{1}{2} f^2 V_\eta e^{-\phi}dA-\ino \frac{1}{2}f^2\L V
e^{-\phi}\dn.
\end{eqnarray}
Taking (\ref{b4}) and (\ref{b5}) into (\ref{b3}), we have
\begin{eqnarray}\label{b6}
\nonumber&&\ino V\(\(\L f+Kn f\)^2-|\n^2 f+K f
g|^2\)\dm\\\nonumber&=& \inpo \Bigg(-V\si f_{i\eta}f_i+ V_\eta |\n
f|^2+V\L f f_\eta+(2n-2)KVff_\eta-f_\eta \si V_i
f_i\\\nonumber&&-(n-1)K f^2 V_\eta\Bigg)
e^{-\phi}dA\\\nonumber&&+\ino \(-\frac{1}{2}\L V|\n f|^2- \sij
V_{ij}f_if_j-\frac{1}{2}\L V|\n f|^2+V\Ric^{\phi}\(\n f, \n
f\)\)e^{-\phi}\dn\\\nonumber&& -(2n-2)K\ino\(V|\n
f|^2-\frac{1}{2}f^2\L V\)e^{-\phi}\dn+n(n-1)K^2\ino V
f^2e^{-\phi}\dn\\&&-2K\ino V f\si f_i \phi_i e^{-\phi}\dn.
\end{eqnarray}
Choosing an orthonormal frame $\{\bar{e}_i\}_{i=1}^n$ such that
$\bar{e}_n=\vec{\eta}$ on $\p \Omega$. Note that $z=f\big|_{\p
\Omega},u=f_\eta\big|_{\p \Omega}$ and
$H^\phi=H+\frac{1}{n-1}\phi_\eta$, we infer from the
Gauss-Weingarten formula that
\begin{eqnarray}\label{b7}
\nonumber&&\inpo V\(\L f f_\eta-\si
f_{i\eta}f_i\)e^{-\phi}dA\\\nonumber&=&\inpo V\(\si
(f_{ii}+f_i\phi_i) f_\eta-\si
f_{i\eta}f_i\)e^{-\phi}dA\\\nonumber&=& \inpo V\(\sum_{i=1}^{n-1}
(f_{ii}+f_i\phi_i) f_\eta+f_\eta^2\phi_\eta-\sum_{i=1}^{n-1}
f_{i\eta}f_i\)e^{-\phi}dA\\&=& \inpo V\(u \overline{\L}
z+(n-1)H^\phi
u^2-\langle\overline{\n}u,\overline{\n}z\rangle+II(\overline{\n}z,\overline{\n}z)\)e^{-\phi}dA
\end{eqnarray}
and
\begin{eqnarray}\label{b8}
\nonumber&&\inpo \(|\n f|^2 V_\eta-\si f_\eta V_i
f_i\)e^{-\phi}dA\\\nonumber&=&\inpo \(|\n z|^2
V_\eta-u\langle\overline{\n} V,\overline{\n}
z\rangle\)e^{-\phi}dA\\&=&\inpo \(|\n z|^2
V_\eta+V\langle\overline{\n} u,\overline{\n} z\rangle+V u
\overline{\L }z\)e^{-\phi}dA.
\end{eqnarray}
Here, in fact, the classical Gauss-Weingarten formula has been used
in (\ref{b7}), that is, for smooth tangent vector fields
$\mathbb{X}$, $\mathbb{Y}$, one has
\begin{eqnarray*}
\nabla_{\mathbb{X}}\mathbb{Y}=\overline{\nabla}_{\mathbb{X}|_{\partial\Omega}}\mathbb{Y}|_{\partial\Omega}+II(\mathbb{X}|_{\partial\Omega},\mathbb{Y}|_{\partial\Omega}),
\end{eqnarray*}
where, as before, $II(\cdot,\cdot)$ stands for the second
fundamental form of the boundary $\partial\Omega$, and  $\nabla$,
$\overline{\nabla}$ are the gradient operators on $\Omega$ and
$\partial\Omega$, respectively. Combining (\ref{b7}) and (\ref{b8}),
we have
\begin{eqnarray}\label{b9}
\nonumber&&\inpo \Bigg(-V\si f_{i\eta}f_i+ V_\eta |\n f|^2+V\L f
f_\eta+(2n-2)KVff_\eta\\\nonumber
 &&\qquad\qquad -f_\eta \si V_i f_i-(n-1)K f^2 V_\eta\Bigg)dA_\phi
\\\nonumber&=&\inpo V\(2u \overline{\L} z+(n-1)H^\phi u^2+II(\overline{\n}z,\overline{\n}z)+(2n-2)Kuz\)dA_\phi\\&&+\inpo V_\eta\(|\overline{\n} z|^2-(n-1)Kz^2\)dA_\phi.
\end{eqnarray}
Substituting (\ref{b9}) into (\ref{b7}), we have
\begin{eqnarray*}
\nonumber&&\ino V\(\(\L f+Kn f\)^2-|\n^2 f+K f g|^2+2K f\langle\n f,\n \phi\rangle\)\dm
\\\nonumber&=&\inpo V\(2u \overline{\L} z+(n-1)H^\phi u^2+II(\overline{\n}z,\overline{\n}z)+(2n-2)Kuz\)dA_\phi\\\nonumber&&+\inpo V_\eta\(|\overline{\n} z|^2-(n-1)Kz^2\)dA_\phi+\ino (n-1)\(K\L V +nK^2 V\)f^2\dm\\&&+\ino \(\n^2 V-\L V g-(2n-2)KV g+V\Ric^\phi\)\(\n f,\n
f\)\dm,
\end{eqnarray*}
which completes the proof of Theorem \ref{theorem1}.
\end{proof}

We need the following generalized Pohozaev identity of the drifting
Laplacian.

\begin{lemma}
Let $F \in \Gamma(T\Omega)$ be a Lipschitz vector field. Let $u \in
H^2(\Omega)$ with $\L u =0$ in $\Omega$ . Then
\begin{eqnarray}\label{c1}
\int_{\partial\Omega}\(\frac{\p u}{\p\vec{\eta}}\cdot g(F,\n
u)-\frac{1}{2}|\n u|^2
g(F,\vec{\eta})\)dA_\phi=\int_{\Omega}\(g(\n_{\n u}F,\n
u)-\frac{1}{2}|\n u|^2\div_\phi F\)\dm,
\end{eqnarray}
where $\div_{\phi}:=\div-g(\cdot,\n\phi)$ denotes the weighted
divergence operator on $\Omega$, and other notations have the same
meanings as before.
\end{lemma}

\begin{proof}
Since $\L u =0$ in $\Omega$, we have
\begin{eqnarray}\label{c2}
\nonumber0&=&\int_{\Omega} \L u \cdot  g(F,\n
u)\dm=\int_{\partial\Omega} \frac{\p u}{\p\vec{\eta}}\cdot g(F,\n
u)dA_\phi-\int_{\Omega}g(\n u, \n g(F,\n u))\dm
\\&=&\int_{\partial\Omega}  \frac{\p u}{\p\vec{\eta}}\cdot g(F,\n
u) dA_\phi-\int_{\Omega} g(\n_{\n u}F,\n u)\dm-\int_{\Omega}\n^2
u(F,\n u)\dm.
\end{eqnarray}
By a direct calculation in an orthonormal local frame chosen for the
tangent bundle $T\Omega$, one has
\begin{eqnarray*}
\n^2 u(F,\n u)=u_{ij}F_iu_j=\(u_jF_iu_j\)_i-u_jF_{i,i}u_j-u_jF_iu_{ji}\\
=\div(|\n u|^2F)-|\n u|^2\div F-\n^2 u(F,\n u).
\end{eqnarray*}
Then, we infer from the integration by parts that
\begin{eqnarray}\label{c3}
\nonumber\int_{\Omega} \n^2 u(F,\n
u)\dm&=&\frac{1}{2}\int_{\partial\Omega} |\n u|^{2}\cdot
g(F,\vec{\eta}) dA_\phi+\frac{1}{2}\int_{\Omega}|\n u|^2
g(F,\n\phi)\dm-\frac{1}{2}\int_{\Omega}|\n u|^2\div
F\dm\\&=&\frac{1}{2}\int_{\partial\Omega} |\n u|^{2}\cdot
g(F,\vec{\eta}) dA_\phi-\frac{1}{2}\int_{\Omega}|\n u|^2\div_{\phi}
F\dm.
\end{eqnarray}
Then (\ref{c1}) follows by substituting (\ref{c3}) into (\ref{c2})
directly.
\end{proof}

In order to prove Theorem \ref{theorem2}, we need to choose a
special function $V=\rho-\frac{c}{2}\rho^2$ in the Reilly-type
formula, where $\rho=\mathrm{dist}(\cdot,\p \Omega)$ denotes the
distance function to the boundary $\p\Omega$, and the constant $c$
is the positive lower bound of the principal curvatures of $\p
\Omega$. Besides, we also need the following fact:

\begin{lemma} (\cite[Proposition 10]{XX}) \label{lemma3-2} For a given compact domain $\Omega$ of the $n$-dimensional
($n\geq2$) SMMS $(M,g,e^{-\phi}dv)$, assume that
$\mathrm{Sec}(\cdot,\cdot)\geq0$ holds on $\Omega$. Fix a
neighborhood $U$ of $\mathrm{Cut}(\p \Omega)$ in $\Omega$, with
$\mathrm{Cut}(\p \Omega)$ the cut-locus of points at the boundary
$\p \Omega$. Then for any $\varepsilon>0$, there exists a smooth
nonnegative function $V_{\varepsilon}$ on $\Omega$ such that
$V_\varepsilon=V$ on $\Omega\setminus U$ and
\begin{eqnarray*}
\n^2 (-V_\varepsilon)\geq(c-\varepsilon)g.
\end{eqnarray*}
\end{lemma}

\vspace{1mm}

\begin{proof} [Proof of Theorem \ref{theorem2}]
 Since $V_\varepsilon|_{\p \Omega}=V|_{\p \Omega}=0$ and $\n_{\vec{\eta}} V_\varepsilon|_{\p  \Omega}=\n_{\vec{\eta}} V|_{\p
 \Omega}=-1$,
then taking $V_{\varepsilon}$ into the Reilly-type formula
(\ref{rf-11}) for $K=0$ we have
\begin{eqnarray} \label{3-2-1}
-\int_{\Omega} V_\varepsilon|\n ^2 f|^2\dm=-\int_{\partial\Omega}
|\overline{\n}z |^2dA_\phi+\int_{\Omega} \(\n ^2 V_\varepsilon-\L
V_\varepsilon g+V_\varepsilon \Ric^\phi\)\(\n f,\n f\)\dm.
\end{eqnarray}
Taking $F=\n V_\varepsilon$ into the generalized Pohozaev identity
(\ref{c1}), we have
\begin{eqnarray} \label{3-2-2}
\int_{\partial\Omega} |\overline{\n}z
|^2dA_\phi=\int_{\partial\Omega}\(\frac{\p
f}{\p\vec{\eta}}\)^2dA_\phi+\int_{\Omega} \(2\n ^2 V_\varepsilon-\L
V_\varepsilon g\)\(\n f,\n f\)\dm.
\end{eqnarray}
Combining (\ref{3-2-1}) and (\ref{3-2-2}) results in
\begin{eqnarray} \label{3-2-3}
\int_{\partial\Omega}\(\frac{\p f}{\p\vec{\eta}}\)^2dA_\phi=
\int_{\Omega} \(-\n ^2 V_\varepsilon+V_\varepsilon|\n ^2
f|^2+V_\varepsilon \Ric^\phi\)\(\n f,\n f\)\dm.
\end{eqnarray}
Putting the assumptions $\mathrm{Sec}(\Omega)\geq0$,
$\mathrm{Ric}(\cdot,\cdot)\leq\kappa$,
$\mathrm{Ric}^{\phi}(\cdot,\cdot)\geq\kappa$, and $II>c\mathrm{I}$
into (\ref{3-2-3}), and using Lemma \ref{lemma3-2}, we can obtain
\begin{eqnarray*}
\int_{\partial\Omega} \(\frac{\p f}{\p\vec{\eta}}\)^2dA_\phi\geq
(c-\varepsilon)\int_{\Omega}  |\n f|^2\dm.
\end{eqnarray*}
  Let $\varepsilon\rightarrow 0$, we have
\begin{eqnarray} \label{3-2-4}
\int_{\partial\Omega} \(\frac{\p f}{\p\vec{\eta}}\)^2dA_\phi\geq
c\int_{\Omega} |\n f|^2\dm.
\end{eqnarray}
Choosing furthermore $f$ to be an eigenfunction corresponding to
$\sigma_{1}(\Omega)$, and then together with (\ref{cha-3}) and
(\ref{3-2-4}), it follows that
 \begin{eqnarray*}
\left(\sigma_{1,\phi}(\Omega)\right)^{2}\int_{\partial\Omega}
f^2dA_\phi=\int_{\partial\Omega} \(\frac{\p
f}{\p\vec{\eta}}\)^2dA_\phi\geq c\int_{\Omega} |\n
f|^2\dm=c\sigma_{1,\phi}(\Omega)\int_{\partial\Omega} f^2dA_\phi,
\end{eqnarray*}
which implies that $\sigma_{1,\phi}(\Omega)\geq c$. When
$\sigma_{1,\phi}(\Omega)= c$, then by \cite[Propositions 15 and
16]{XX}, (\ref{3-2-3}) and (\ref{3-2-4}), we know that
$\nabla^{2}\phi=0$ and $\Omega$ is isometric to an $n$-dimensional
Euclidean ball of radius $\frac{1}{c}$. This finishes the proof of
Theorem \ref{theorem2}.
\end{proof}

\section*{Appendix A}
\renewcommand{\thesection}{\arabic{section}}
\renewcommand{\theequation}{\thesection.\arabic{equation}}
\setcounter{equation}{0}

In this part, we would like to give a detailed explanation for the
spectrum of (\ref{a10}), and this explanation is inspired by an
argument shown in \cite[pp. 2202-2204]{AM}.

Let $\mathcal{H}$ be an infinite dimensional separable Hilbert
space, and let $\mathcal{V}$ be another Hilbert space, which is
embedded as a dense subspace in $\mathcal{H}$, so that
$\mathcal{V}\subset\mathcal{H}\subset\mathcal{V}^{\ast}$. Assume
that $a$ is a closed, symmetric, real-valued, coercive quadratic
form, i.e., for all $u\in\mathcal{V}$, the inequality
\begin{eqnarray*}
a(u)+w\|u\|^{2}_{\mathcal{H}}\geq\alpha\|u\|^{2}_{\mathcal{V}}
\end{eqnarray*}
holds for some $w\in\mathbb{R}$ and $\alpha>0$. Let
$\mathcal{D}:\mathcal{V}\rightarrow\mathcal{V}^{\ast}$ be the
Dirichlet-to-Neumann map. Associated to $a$ is a bounded operator
$A_{1}:\mathcal{V}\rightarrow\mathcal{V}^{\ast}$. Also associated to
$a$ is an unbounded self-adjoint operator $A_{2}$ on $\mathcal{H}$
with domain $\mathcal{D}(A_2)\subset\mathcal{V}\subset\mathcal{H}$.
Therefore, one has:
\begin{itemize}

\item $x\in\mathcal{D}(A_1)$ and $A_{1}(x)=y\in\mathcal{V}^{\ast}$ if
and only if $a(x,v)=\langle y,v\rangle$ for all $v\in\mathcal{V}$.

\end{itemize}
It is not hard to know the operator $A_2$ is the part of $A_1$ in
$\mathcal{D}(A_2)$, and so one can write either operator as $A$ and
drop the subscript. Based on these preparations, one can obtain:
\begin{itemize}

\item The form $a$ is accretive (i.e., $a(u)\geq0$ for all
$u\in\mathcal{V}$) if and only if $A$ is nonnegative (i.e., $\langle
Au,u\rangle_{\mathcal{H}}\geq0$ for all $u\in\mathcal{D}(A)$);

\item $A$ has compact resolvent, and hence discrete spectrum, if and
only if the inclusion $\mathcal{D}(A)\hookrightarrow\mathcal{H}$ is
compact.

\end{itemize}
Clearly, if $\mathcal{V}\hookrightarrow\mathcal{H}$ is compact, then
one certainly has $\mathcal{D}(A)\hookrightarrow\mathcal{H}$ is
compact and of course $A$ has the discrete spectrum.

In what follows, we prefer to use the above argument related to the
Dirichlet-to-Neumann map $\mathcal{D}$ to show the eigenvalue
problem (\ref{a10}) has discrete spectrum. In fact, one would see
that:
\begin{itemize}
\item \emph{ The eigenvalues of the eigenvalue problem of the
drifting Laplacian (\ref{a10})  can be interpreted as the
eigenvalues of the Dirichlet-to-Neumann operator
$\mathcal{D}:\widetilde{H}^{1/2}(\partial\Omega)\rightarrow
\widetilde{H}^{-1/2}(\partial\Omega)$ which maps a function
$f\in\widetilde{H}^{1/2}(\partial\Omega)$ to
$\mathcal{D}f=\partial_{\vec{\eta}}(H_{w}f)=\frac{\partial
(H_{w}f)}{\partial\vec{\eta}}$, where $H_{w}f$ is the weighted
harmonic extension of $f$ to the interior of $\Omega$ (i.e.,
$\L(H_{w}f)=0$ in $\Omega$).}
\end{itemize}
For the eigenvalue problem (\ref{a10}), since $\partial\Omega$ is
smooth\footnote{In fact, by checking the argument in Appendix A
carefully, one would find that if the boundary $\partial\Omega$ is
only Lipschitz, the eigenvalue problem (\ref{a10}) also has discrete
spectrum.}, $\partial\Omega$ is locally a smooth graph such that
$\Omega$ lies locally on one side of $\partial\Omega$, and w.r.t.
the weighted volume densities $\dm$ and $dA_{\phi}$, one can define
the Hilbert space $\widetilde{L}^{2}(\Omega)$, the Sobolev space
$\widetilde{H}^{1}(\Omega)$ and $\widetilde{L}^{2}(\partial\Omega)$
naturally. More precisely, $\widetilde{L}^{2}(\Omega)$ is the set of
functions defined over $\Omega$ and satisfying
$\int_{\Omega}f^{2}\dm=\int_{\Omega}f^{2}e^{-\phi}dv<\infty$,
$\widetilde{L}^{2}(\partial\Omega)$ is the set of functions defined
over $\partial\Omega$ and satisfying
$\int_{\partial\Omega}f^{2}dA_{\phi}=\int_{\partial\Omega}f^{2}e^{-\phi}dA<\infty$,
and as mentioned before\footnote{Here we have used a convention --
for the Sobolev space $\widetilde{W}^{k,p}(\cdot)$ w.r.t. the
weighted volume density, if $p=2$, we usually write
$\widetilde{H}^{k}(\cdot)=\widetilde{W}^{k,p}(\cdot)$,
$k=0,1,2,\cdots$.}
$\widetilde{H}^{1}(\Omega):=\widetilde{W}^{1,2}(\Omega)$ is the
completion of the set of smooth functions $C^{\infty}(\Omega)$ under
the weighted Sobolev norm ${\widetilde{\|\cdot\|}}_{1,2}$. Clearly,
one can naturally define an inner product
$\langle\cdot,\cdot\rangle$
 in  $\widetilde{L}^{2}(\Omega)$ as follows
 \begin{eqnarray*}
 \langle f,h\rangle=\int_{\Omega}fh\dm,
 \end{eqnarray*}
which leads to the fact that $\widetilde{L}^{2}(\Omega)$ becomes a
Hilbert space. Similarly, $\widetilde{H}^{1}(\Omega)$,
$\widetilde{L}^{2}(\partial\Omega)$ would be Hilbert spaces by
suitably defining inner products. Let $C^{\infty}_{0}(\Omega)$ be
the set of smooth functions (defined on $\Omega$) with compact
support, and $\widetilde{H}^{1}_{0}(\Omega)$ be the closure of
$C^{\infty}_{0}(\Omega)$ in $\widetilde{H}^{1}(\Omega)$ under the
weighted norm. The boundary restriction map $u\mapsto
u|_{\partial\Omega}:=\mathcal{R}_{m}u$ is well-defined for any
$u\in\widetilde{H}^{1}(\Omega)\cap C^{0}(\overline{\Omega})$, and
this map can be extended to a bounded operator
$\mathcal{R}_{m}:\widetilde{H}^{1}(\Omega)\rightarrow\widetilde{L}^{2}(\partial\Omega)$
with nullspace $\widetilde{H}^{1}_{0}(\Omega)$. We have the
following facts:
\begin{itemize}
\item Assume that $u\in\widetilde{H}^{1}(\Omega)$. One says that $\L
u\in\widetilde{L}^{2}(\Omega)$ if there exists
$f\in\widetilde{L}^{2}(\Omega)$ such that
\begin{eqnarray*}
 \int_{\Omega}\nabla u \cdot \overline{\nabla
 v}\dm=\int_{\Omega}f\overline{v}\dm \qquad
 \mathrm{for~all}~v\in\widetilde{H}^{1}_{0}(\Omega).
\end{eqnarray*}

\item Assume that $u\in\widetilde{H}^{1}(\Omega)$ and $\L
u\in\widetilde{L}^{2}(\Omega)$. One says that
$\partial_{\vec{\eta}}u\in\widetilde{L}^{2}(\partial\Omega)$ if
there exists $h\in\widetilde{L}^{2}(\partial\Omega)$ such that
 \begin{eqnarray*}
\int_{\Omega}\left(\nabla u \cdot \overline{\nabla
 v}-\L
 u\cdot\overline{v}\right)\dm=\int_{\partial\Omega}h\overline{v}dA_{\phi}\qquad
 \mathrm{for~all}~v\in\widetilde{H}^{1}(\Omega),
 \end{eqnarray*}
and then one writes $\partial_{\vec{\eta}}u=h$.
\end{itemize}
Here and later, for simplification, we often omit the trace signs
under the integral, e.g., simply write
$\int_{\partial\Omega}hvdA_{\phi}=\int_{\partial\Omega}hv|_{\partial\Omega}dA_{\phi}$.
Based on these facts, it is not hard to get the following Green's
formula:
\begin{eqnarray*}
\int_{\Omega}\left(\nabla u \cdot \overline{\nabla
 v}-\L
 u\cdot\overline{v}\right)\dm=\int_{\partial\Omega}\left(\partial_{\vec{\eta}}u\right)\cdot\overline{v}dA_{\phi}
\end{eqnarray*}
holds for all $v\in\widetilde{H}^{1}(\Omega)$ whenever
$u\in\widetilde{H}^{1}(\Omega)$, $\L u\in\widetilde{L}^{2}(\Omega)$
and $\partial_{\vec{\eta}}u\in\widetilde{L}^{2}(\partial\Omega)$.

For any $\sigma\in\mathbb{R}$, consider the quadratic form
\begin{eqnarray*}
b_{\sigma}(u,v):=\int_{\Omega}\nabla u \cdot \overline{\nabla
 v}\dm-\sigma\int_{\partial\Omega}u\cdot\overline{v}dA_{\phi}, \qquad
 \mathrm{for}~u,v\in\widetilde{H}^{1}(\Omega).
\end{eqnarray*}
By Green's formula, one has
\begin{eqnarray*}
b_{\sigma}(u,u)&=&\int_{\Omega}\nabla u \cdot \overline{\nabla
 u}\dm-\sigma\int_{\partial\Omega}u\cdot\overline{u}dA_{\phi}\\
 &=&\int_{\Omega}\left(\L u\right)u\dm+\int_{\partial\Omega}\left(\partial_{\vec{\eta}}u\right)\cdot udA_{\phi}-\sigma\int_{\partial\Omega}u^{2}dA_{\phi}
\end{eqnarray*}
for $u\in\widetilde{H}^{1}(\Omega)$. Hence, the form
$b_{\sigma}(u,v)$ associated with the boundary value problem
(\ref{a10}) should be coercive since $b_{\sigma}(u,u)\geq0$ for
$u\in\widetilde{H}^{1}(\Omega)$. So, this quadratic form
$b_{\sigma}(\cdot,\cdot)$ determines an operator $(\L)_{\sigma}$,
and letting $v\in\widetilde{H}^{1}_{0}(\Omega)$ shows that
$(\L)_{\sigma}$ should be the drifting Laplacian $\L$ in the
interior of $\Omega$. Then we can infer that
$u\in\mathcal{D}((\L)_{\sigma})$ implies
$\partial_{\vec{\eta}}u=\sigma u$, at least in the weak sense.
Therefore, one has
\begin{eqnarray*}
\mathcal{D}((\L)_{\sigma})=\left\{u\in\widetilde{L}^{2}(\Omega)\Big{|}\L
u\in\widetilde{L}^{2}(\Omega),\partial_{\vec{\eta}}u~\mathrm{exists~and~}\partial_{\vec{\eta}}u=\sigma
\cdot u|_{\partial\Omega}\right\}.
\end{eqnarray*}
Since $\widetilde{H}^{1}(\Omega)$ is compactly included in
$\widetilde{L}^{2}(\Omega)$, using the facts of Dirichlet-to-Neumann
map shown at the first part of this section, one knows that the
operator $\L$ related to the eigenvalue problem (\ref{a10}) has
discrete spectrum. Once we have this conclusion, the
characterization (\ref{cha-3}) follows by using the standard
argument of the variational method (see, e.g., \cite[Section 5,
Chapter I]{IC} for Rayleigh's theorem and Max-min theorem from which
one can get the characterizations for eigenvalues of the Laplacian
of prescribed type).

\vspace{5mm}

\section*{Acknowledgments}
\renewcommand{\thesection}{\arabic{section}}
\renewcommand{\theequation}{\thesection.\arabic{equation}}
\setcounter{equation}{0} \setcounter{maintheorem}{0}

This research was supported in part by the NSF of China (Grant Nos.
11801496 and 11926352), the Fok Ying-Tung Education Foundation
(China), and Hubei Key Laboratory of Applied Mathematics (Hubei
University). Prof. Y. L. Deng joined us and made essential
contributions to  revisions of the previous version of our
manuscript (i.e., arXiv:2003.13231v1), so we prefer to list him as
one of the co-authors. The authors are grateful to the anonymous
referee for careful reading and valuable comments such that the
paper appears as its present version.

\end{document}